\documentclass[oneside,10pt,reqno]{amsart}
\usepackage{amssymb,amsmath,amsthm,bbm,enumerate,mdwlist,url,multirow,hyperref,amsthm,paralist}
\usepackage[pdftex]{graphicx}
\usepackage[shortlabels]{enumitem}

\theoremstyle{definition}
\newtheorem{definition}{Definition}
\theoremstyle{theorem}

\newtheorem{lemma}[definition]{Lemma}
\newtheorem{theorem}[definition]{Theorem}

\theoremstyle{remark}
\newtheorem{remark}[definition]{Remark}

\newtheorem{question}[definition]{Question}
\newtheorem{example}[definition]{Example}
\newtheorem{notation}[definition]{Notation}
\DeclareMathSymbol{\shortminus}{\mathbin}{AMSa}{"39}

\def\PP{\mathbb{P}}

\def\AA{\mathcal{A}}

\def\HH{\mathcal{H}}

\def\L2{\mathrm{L}^2}
\def\pr{\mathrm{pr}}
\def\QQ{\mathbb{Q}}
\def\BB{\mathcal{B}}

\def\AA{\mathcal{A}}

\def\id{\mathrm{id}}
\def\GG{\mathcal{G}}
\def\FF{\mathcal{F}}

 \def\L2{\mathrm{L}^2}
 
 \def\ZZn{\mathbb{Z}_{\leq 0}}
\makeatletter
\@namedef{subjclassname@2020}{%
  \textup{2020} Mathematics Subject Classification}
\makeatother
\begin{document}
\title{A nonclassical solution to a classical SDE and a converse to Kolmogorov's zero-one law}

\author{Matija Vidmar}
\address{Department of Mathematics, Faculty of Mathematics and Physics, University of Ljubljana, Slovenia}
\email{matija.vidmar@fmf.uni-lj.si}

\begin{abstract}
For a discrete-negative-time discrete-space SDE, which admits no strong solution in the classical sense, a weak solution is constructed that is a (necessarily nonmeasurable) non-anticipative function of the driving i.i.d. noise. The result highlights the strong r\^ole measurability plays in (non-discrete) probability. En route one --- quite literally --- stumbles upon a  converse to the celebrated Kolmogorov's zero-one law for sequences with independent values. 
\end{abstract}

\thanks{Financial support from the Slovenian Research Agency is acknowledged (programme No. P1-0402). 
}



\maketitle

\section{Introduction and main results}
 All filtrations and processes in this section are indexed by $\ZZn$; the natural filtration of a process $Z$ is denoted $\FF^Z$: $\FF^Z_n:=\sigma(Z_m:m\in \mathbb{Z}_{\leq n})$ for $n\in \ZZn$. Consider the following classical (simplest non-trivial) discrete-negative-time discrete-space SDE (stochastic difference equation): 
\begin{equation}\label{main}
X_n=X_{n-1}\xi_n,\quad n\in \ZZn,
\end{equation}
where $(\xi_n)_{n\in\ZZn}$ is a sequence of independent equiprobable random signs [for each $n\in \ZZn$, $\xi_n$ is $\{-1,1\}$-valued and $\PP(\xi_n=1)=\frac{1}{2}$] and where $(X_n)_{n\in \ZZn}$ is the unknown process that also takes its values in $\{-1,1\}$. It is paradigmatic \cite[Eq.~(1)]{yano2015} and indeed intimately related \cite[Eqs.~(25) and (26)]{yano2015} to Tsirelson's ``celebrated and mysterious'' stochastic differential equation  \cite[V.18, p. 155]{rogers2000diffusions}. Let us recall the most conspicuous features of \eqref{main}.

\begin{definition}\label{definition:classical}
\begin{inparaenum}[(a)]
\item\label{a} A weak solution to \eqref{main} consists of a filtered probability space $(\Omega,\GG,\PP,\FF)$ and of a pair  $(\xi,X)$ of $\FF$-adapted $\{-1,1\}$-valued processes defined thereon such that \eqref{main} holds and such that for each $i\in\ZZn$, $\xi_i$ is an equiprobable random sign independent of $\FF_{i-1}$. \item\label{b} A strong solution to \eqref{main} is a weak solution, as in \eqref{a}, for which $\FF^X$ is included in $\FF^\xi$. 
\item Uniqueness in law holds for \eqref{main} if  in any  weak solution from \eqref{a} the process $X$ has the same law.
\end{inparaenum}
\end{definition} 

$(\bullet_1)$ Take a weak solution of Definition~\ref{definition:classical}\eqref{a}. For any $n\in \ZZn$, $\PP(X_n=1)=\PP(X_{n-1}=-1,\xi_n=-1)+\PP(X_{n-1}=1,\xi_n=1)=\PP(X_{n-1}=-1)\PP(\xi_n=-1)+\PP(X_{n-1}=1)\PP(\xi_n=1)=\frac{1}{2}(\PP(X_{n-1}=-1)+\PP(X_{n-1}=1))=\frac{1}{2}$. Therefore, because of \eqref{main} again and because the $\xi_n$, $n\in \mathbb{Z}_{\leq 0}$, are independent equiprobable random signs relative to $\FF$ to which $X$ is adapted, the $X_n$, $n\in \ZZn$, are independent equiprobable random signs also. There is uniqueness in law for \eqref{main}. In particular, by Kolmogorov's zero-one law, the tail $\sigma$-field $\FF^X_{-\infty}:=\cap_{m\in \ZZn}\FF^X_m$ of $X$ in any weak solution of Definition~\ref{definition:classical}\eqref{a} is always trivial (even if one were to complete the filtration $\FF^X$ before taking the intersection of its members, of course).

$(\bullet_2)$  On the other hand, let, on some probability space $(\Omega,\GG,\PP)$, $X=(X_n)_{n\in\ZZn}$ be a sequence of independent equiprobable random signs, $\FF=\FF^X$ its natural filtration, and define the process $\xi=(\xi_n)_{n\in \ZZn}$ so that it  satisfies \eqref{main}. It gives a weak solution of \eqref{main}: plainly the $\xi_n$, $n\in \mathbb{N}$, are equiprobable random signs; furthermore, for all  $n\in \ZZn$ and for all $k\in \mathbb{N}$ one has $\PP(\xi_n=1,X_{n-1}=\cdots=X_{n-k}=1)=\PP(X_n=X_{n-1}=\cdots=X_{n-k}=1)=2^{-k-1}=\PP(\xi_n=1)\PP(X_{n-1}=\cdots=X_{n-k}=1)$, yielding the independence of $\xi_n$ from $\FF^X_{n-1}$ (while the adaptedness of $\xi$ to $\FF^X$ is clear). 

$(\bullet_3)$  Finally, take again any weak solution of Definition~\ref{definition:classical}\eqref{a}. For each $n\in \ZZn$ and for each $k\in \mathbb{N}$ one has $\PP(X_n=1,\xi_n=\cdots=\xi_{n-k+1}=1)=\PP(\xi_n=\cdots=\xi_{n-k+1}=1,X_{n-k}=1)=\PP(\xi_n=\cdots=\xi_{n-k+1}=1)\PP(X_{n-k}=1)=\PP(X_n=1)\PP(\xi_n=\cdots=\xi_{n-k+1}=1)$. Therefore, for all $n\in \ZZn$, $X_n$ is independent of $\FF^\xi_n$ (hence in fact of the whole of $\xi$); being non-degenerate, it cannot also be $\FF^\xi_n$-measurable. No weak solution to \eqref{main} can ever be strong.

\begin{remark}\label{rmk}
In Definition~\ref{definition:classical} one could ask, ceteris paribus: in \eqref{a} for \eqref{main} to hold only a.s.-$\PP$; 
and/or in \eqref{b} for $\FF^X$ to be included only in the $\PP$-completion of $\FF^\xi$. It would be without consequence for $(\bullet_1)$-$(\bullet_2)$-$(\bullet_3)$. 
\end{remark}

The preceding is well-known --- the multiplicative-increments-evolution process $\xi$ of $X$ in \eqref{main} innovates but  fails to generate $X$ (even though the tail $\sigma$-field $\FF^X_{-\infty}$ of $X$ is trivial!): in no weak solution of Definition~\ref{definition:classical}\eqref{a} can any of the $X_n$, $n\in \ZZn$, be a measurable function of $\xi$. In the phrasing of  \cite{emery} ``the answer to the innovation problem [for \eqref{main}] is negative, some kind of creation of information occurs'', the extra information  ``appears magically, from thin air'' \cite{documenta}. But nevertheless (to the best of the author's knowledge, a novel result), 
\begin{theorem}\label{theorem}
 \eqref{main} admits a weak solution of Definition~\ref{definition:classical}\eqref{a} in which, for each $n\in\ZZn$,  $X_n$ is a function of $\xi\vert_{\mathbb{Z}_{\leq n}}$ [necessarily this function is not measurable w.r.t. $(2^{\{-1,1\}})^{\otimes \mathbb{Z}_{\leq n}}$, of course]; 
\end{theorem}
\noindent in \eqref{main} the evolution process \emph{can} explain everything (albeit non-measurably)!
It is shown to be true in Section~\ref{sec:construction}. Remark however already here that 

$(\bullet_4)$ \eqref{main} admits also a weak solution of Definition~\ref{definition:classical}\eqref{a} in which the property of Theorem~\ref{theorem} fails on every $\PP$-almost certain set (so the statement of the theorem is not trivial). Take indeed the solution of $(\bullet_2)$ with $\Omega=\{-1,1\}^{\ZZn}$, $X$ the coordinate projections, $\GG=\FF_0$ (hence $\PP=(\frac{1}{2}\delta_{-1}+\frac{1}{2}\delta_1)^{\times \ZZn}$). Let $\Omega^*$ be $\PP$-almost certain. Put $\Omega^{**}:=\Omega^*\backslash \theta(\Omega\backslash \Omega^{*})$, where $\theta:=-\id_\Omega$ is the involutive (its square is the identity) measure-preserving transformation of $\Omega$ that flips all the signs. Then $\Omega^{**}\in 2^{\Omega^*}$ is $\PP$-almost certain and  $\theta(\Omega^{**})=\theta(\Omega^*)\backslash (\Omega\backslash \Omega^{*})\subset \Omega^*$. Take any $\omega\in \Omega^{**}$ (it exists); then $\{\omega,\theta(\omega)\}\subset \Omega^{*}$. One has $\xi(\omega)=\xi(\theta(\omega))$, while $X_k(\omega)\ne -X_k(\omega)= X_k(\theta(\omega))$ for all $k\in \ZZn$. So, in fact, on  no $\PP$-almost certain $\Omega^*$ can any of the $X_k$, $k\in \ZZn$, be a function of $\xi$. 

The result of Theorem~\ref{theorem} brings to the forefront the significant r\^ole that measurability actually plays in non-discrete\footnote{Why non-discrete? True, the random signs (of a weak solution $(\xi,X)$ of Definition~\ref{definition:classical}\eqref{a}) \emph{individually} are discrete. However, their totality is not.} probability, and that is perhaps sometimes not so clearly visible -- though of course not unappreciated in the literature, see e.g. \cite{vovk} for a relatively recent study in the context of game-theoretic probability. An analogue of Theorem~\ref{theorem} for the case of random elements with diffuse laws is provided in Remark~\ref{continuous-space-analogue}.

\begin{question}
Could one also observe the same basic phenomenon in continuous time (and space)? In particular for Tsirelson's stochastic differential equation? The answer is probably yes, but such construction appears nevertheless to be  more involved.
\end{question}

In passing to Theorem~\ref{theorem} one finds informative (albeit a very special case of) the following converse to Kolmogorov's zero-one law. To better appreciate it, the reader will recall the content of the latter: if $(\xi_i)_{i\in I}$ is any independency of sub-$\sigma$-fields under a probability $\PP$, then $\limsup \xi:=\bigcap_{\text{finite }\!F\in 2^I}\lor_{i\in I\backslash F}\xi_i\subset \PP^{-1}(\{0,1\})$; in particular the tail $\sigma$-field of a sequence of independent random elements  is trivial. What the result to follow shows is that, in the discrete setting, a kind of (the best one can hope for) converse also holds: except when this obviously fails, an event of a sequence with independent values is negligible (resp. almost certain) only if it is contained in a negligible (resp. contains an almost certain) tail event of said sequence. 

\begin{theorem}\label{thm:kolmo}
Let $(\Omega,\GG,\PP)$ be a probability space and let $\xi=(\xi_n)_{n\in \mathbb{N}}$ be a sequence of independent random elements thereon with $\xi_n$ valued in a countable set $E_n$ for $n\in \mathbb{N}$. 
Consider the following statements.
\begin{enumerate}[(i)]
\item\label{Ki}  For all $n\in \mathbb{N}$ and $e\in E_n$, $\PP(\xi_n=e)>0$.
\item\label{Kii} For every $\PP$-a.s. $\Omega^*\in \sigma(\xi)$  there exists a $\PP$-a.s. $\Omega^{**}\in \limsup_{n\to\infty}\sigma(\xi_n)$ with $\Omega^{**}\subset \Omega^*$.
\item\label{Kiii} For every $\PP$-negligible $\Omega^*\in \sigma(\xi)$  there exists a $\PP$-negligible $\Omega^{**}\in \limsup_{n\to\infty}\sigma(\xi_n)$ with $\Omega^{**}\supset \Omega^*$.
\end{enumerate}
Then \ref{Kii} and \ref{Kiii} are equivalent, and they are implied by \ref{Ki}. If furthermore $\xi$ is sufficiently nice in the sense that
\begin{quote} for all $n\in \mathbb{N}$, for all $\{e,f\}\subset E_n$ with $e\ne f$ and $\PP(\xi_n=e)=0$, and for all $\omega\in \Omega$ with $\xi_n(\omega)=f$, there exist an $\omega'\in \Omega$ and a $k\in \mathbb{N}$ such that $\xi_n(\omega')=e$ while $\xi_l(\omega)=\xi_l(\omega')$ for all $l\in \mathbb{N}_{\geq k}$,
\end{quote}
then the statements \ref{Ki}-\ref{Kii}-\ref{Kiii} are in fact all equivalent. 
\end{theorem}
Theorem~\ref{thm:kolmo}, the second main finding of this paper, is proved in Section~\ref{sec:kolmo}. Surprisingly, the result does not appear to have been noted in the literature thus far, though we may mention a counterexample on a would-be converse to Kolmogorov's zero-one law in another direction: triviality of the tail $\sigma$-field does not require independence \cite[1.24]{romano1986counterexamples}. Some immediate remarks concerning Theorem~\ref{thm:kolmo} are as follows.
\begin{enumerate}[(a)]
\item $\xi$ is certainly ``sufficiently nice'' if it is the canonical process on $\prod_{m\in\mathbb{N}}E_m$; as usual the main added value  --- viz. ``sitting'' oneself on  a canonical space --- of this, arguably very technical, condition displayed in Theorem~\ref{thm:kolmo} appears to be in it being able to handle spaces that are products of the canonical space and some other space. 
\item Perhaps one could weaken the ``$\xi$ is sufficiently nice'' condition, but one cannot dispense with it completely, simply because, waiving it, then any of the $E_n$, $n\in \mathbb{N}$, can be enlarged by some $e'\notin E_n$, without affecting the validity of \ref{Kii} or \ref{Kiii}, while  $\PP(\xi_n=e')=0$ for such $e'$. Of course in the preceding the equivalence of \ref{Ki} and \ref{Kii} fails somehow for trivial reasons; see however Example~\ref{example:suff-nice} for a more satisfying counterexample. 
\item The countability of the ranges of the $\xi_n$, $n\in \mathbb{N}$, is, apparently, more or less essential for anything of interest to be recorded in this vein (see  Remark~\ref{remark:ceteris-paribus}).
\item Instead of with the sequence of discrete random elements $\xi$ one could work, in a clear way, with a sequence of countable measurable partitions. However, it seems easier to think about the matter in terms of sequences of random elements.
\item The independence assumption of Theorem~\ref{thm:kolmo} is essential, see Example~\ref{ex:indep}. 
\item By discarding a $\PP$-negligible event and making the $E_n$, $n\in \mathbb{N}$, smaller if necessary, condition \ref{Ki} can always be forced if it does not hold to begin with. 
\end{enumerate} 
For a ``positive testament'' to Theorem~\ref{thm:kolmo} see Example~\ref{example:last}. 

Finally, as it is perhaps slightly nonstandard, before proceeding to the proofs, let us make it explicit that
\begin{notation}
we will write: $\AA/\BB$ for the set of $\AA/\BB$-measurable maps; $\overline{\AA}^\QQ:=\AA\lor \QQ^{-1}(\{0\})$ for the $\QQ$-completion of $\AA$; $\overline{\QQ}$ for the completion of $\QQ$.
\end{notation}

\section{Theorem~\ref{theorem}: construction of a non-anticipative solution to \eqref{main}}\label{sec:construction}
It will be more convenient in this section to work with $\mathbb{N}$ in lieu of $\ZZn$ as the (temporal) index set.

Let $\Omega:=\{-1,1\}^{\mathbb{N}}$, $\xi=(\xi_n)_{n\in \mathbb{N}}$ the coordinate process on $\Omega$, $\sim$ the equivalence relation of equality of tails:
$$\omega_1\sim\omega_2\Leftrightarrow \left(\omega_1=\omega_2\text{ on }\mathbb{N}_{\geq n}\text{ for some $n\in \mathbb{N}$}\right),\quad \{\omega_1,\omega_2\}\subset \Omega.$$ Let also $\Omega^*$ be the range of a choice function on $\Omega/_\sim$; assume for convenience (as one may) that $\mathbbm{1}_\mathbb{N}\in \Omega^*$. 

For $\omega^*\in \Omega^*$ put $X_1(\omega^*):=1$ and then inductively $X_{n+1}(\omega^*):=X_n(\omega^*)\xi_n(\omega^*)$ for $n\in \mathbb{N}$ [in particular $X_n(\mathbbm{1}_\mathbb{N})=1$ for all $n\in \mathbb{N}$];  for $\omega\in \Omega\backslash \Omega^*$ let $\omega^*$ be the unique element of $\Omega^*$ equivalent to $\omega$, let $n\in \mathbb{N}$ be such that $\omega=\omega^*$ on $\mathbb{N}_{\geq n}$ [there is ambiguity in $n$, but it does not matter], put $X_n(\omega):=X_n(\omega^*)$ and define $X_k(\omega)$ for $k\in \mathbb{N}\backslash \{n\}$ so that the recursion 
$$X_{l+1}(\omega)=X_{l}(\omega)\xi_l(\omega),\quad l\in \mathbb{N},$$ is satisfied (it holds for $\omega\in \Omega^*$ also). For  each $n\in \mathbb{N}$, $X_n$ is a function of $(\xi_k)_{k\in \mathbb{N}_{\geq n}}$: if $\xi_k(\omega)=\xi_k(\omega')$ for all $k\in \mathbb{N}_{\geq n}$, then $\omega\sim\omega'$ and (so) $X_n(\omega)=X_n(\omega')$, no matter what the $\omega$ and $\omega'$ from $\Omega$ may be. The preceding construction is due to Jon Warren \cite{warren}.

Let now $\PP:=(\frac{1}{2}\delta_{-1}+\frac{1}{2}\delta_1)^{\times \mathbb{N}}$ be the ``fair-coin-tossing'' measure on $\mathcal{B}_\Omega:=(2^{\{-1,1\}})^{\otimes \mathbb{N}}$. Note that $\mathcal{B}_\Omega$ is also the Borel $\sigma$-field on $\Omega$ for the product topology (where each coordinate has the discrete topology) and that the map $\Phi:=(\Omega\ni \omega\mapsto \sum_{n\in \mathbb{N}}\frac{\omega(n)+1}{2^{n+1}}\in [0,1] )$ is  continuous as well as a mod-$0$ isomorphism between $\overline{\PP}$, the completion of $\PP$, and the Lebesgue measure on $[0,1]$. 
Under $\PP$ the random variables $\xi_n$, $n\in \mathbb{N}$, are independent equiprobable random signs. 


Now, none of the $X_n$, $n\in \mathbb{N}$, is a random variable under $\PP$ (meaning that none of them is $\mathcal{B}_\Omega$-measurable). For if it was, then each of the $X_n$, $n\in \mathbb{N}$, would be so, and then, again for each $n\in \mathbb{N}$, because $X_n$ is a function of $(\xi_k)_{k\in \mathbb{N}_{\geq n}}$, it would even be a $(2^{\{-1,1\}})^{\otimes  \mathbb{N}_{\geq n}}$-measurable function of the $(\xi_k)_{k\in \mathbb{N}_{\geq n}}$ [this is because of the structure of the space; quite simply  
$X_n=X_n(\psi_n)$, where $\psi_n(\omega):=(\underbrace{1,\ldots,1}_{(n\shortminus 1)\text{times}},\xi\vert_{\mathbb{N}_{\geq n}}(\omega))$ for $\omega\in \Omega$], which in turn, upon a trivial transposition from $\mathbb{N}$ to $\ZZn$, would yield a strong solution to \eqref{main}, a contradiction (recall $(\bullet_3)$ from the Introduction). 

In fact, for each $n\in \mathbb{N}$, $X_n$ is not even a random variable under $\overline{\PP}$ (i.e. not $\overline{\mathcal{B}_\Omega}^{\overline{\PP}}$-measurable): a simple completion cannot (begin to) save us. It is not unexpected, though it is a little less obvious. To see it we proceed yet again by contradiction. If one (equivalently each) of the $X_n$, $n\in \mathbb{N}$, would be a random variable under $\overline{\PP}$, then, for all $n\in \mathbb{N}$, $X_n=X_n'$ a.s.-$\overline{\PP}$ for some $X_n'\in \mathcal{B}_\Omega/2^{\{-1,1\}}$. Thus, by Theorem~\ref{thm:kolmo} (its proof will of course be independent of this argument), on a $\PP$-almost certain tail event  $A$ of $\xi$, we would have $X_n=X_n'$ and hence $X_n=X_n'(\psi_n)$ for all $n\in \mathbb{N}$ [the tail event $A$ intervenes somewhat crucially here: for $\omega\in A$ also $\psi_n(\omega)\in A$ (because $A\in \sigma(\xi\vert_{\mathbb{N}_{\geq n}})$), hence  $X_n'(\psi_n(\omega))=X_n(\psi_n(\omega))=X_n(\omega)$ (because $X_n$ is a function of $\xi\vert_{\mathbb{N}_{\geq n}}$)], rendering $X_n\in \overline{\sigma(\xi\vert_{\mathbb{N}_{\geq n}})}^{\overline{\PP}}/2^{\{-1,1\}}$.\footnote{Of course since $X_n=X_n\circ \psi_n$, to show that $X_n=X'_n\circ \psi_n$ a.s.-$\overline{\PP}$ (and hence that $X_n\in \overline{\sigma(\xi\vert_{\mathbb{N}_{\geq n}})}^{\overline{\PP}}/2^{\{-1,1\}}$), really one needs only that $\overline{\PP}(\psi_n\in \{X_n=X_n'\})=1$; however, $\psi_n$ is not measure-preserving and therefore it is presumably not (entirely) obvious, i.e. an intervention of (something akin to) Theorem~\ref{thm:kolmo} seems necessary.  In view of $X_n=X_n\circ \psi_n$, an alternative path to establishing that $X_n\in \overline{\sigma(\xi\vert_{\mathbb{N}_{\geq n}})}^{\overline{\PP}}/2^{\{-1,1\}}$, would be to argue that $\psi_n\in \overline{\sigma(\xi\vert_{\mathbb{N}_{\geq n}})}^{\overline{\PP}}/\overline{\mathcal{B}_\Omega}^{\overline{\PP}}$; since  $\psi_n\in \sigma(\xi\vert_{\mathbb{N}_{\geq n}})/\mathcal{B}_\Omega$ it amounts to checking that $\psi_n^{-1}(A)$ is $\PP$-negligible when $A$ is, which is basically the same kind of thing as was needed before. On the other hand, it is also clear that the full force of Theorem~\ref{thm:kolmo} is not needed here, and in lieu of it one could certainly provide a --- shorter when compared to the proof of Theorem~\ref{thm:kolmo} --- argument tailored to this specific context.} But then we would again obtain a strong solution to \eqref{main} (recall Remark~\ref{rmk}), a contradiction. (There are many other interesting constructions of non-measurable sets from a sequence of /independent/ coin tosses, e.g. \cite{blackwell,soo}.)

In spite of the preceding, as we shall see, we will be able to extend $\PP$ to a probability $\PP'$ in such a manner that, under $\PP'$, $X_1$ is an equiprobable random sign independent of $\xi$. Then, plainly, under $\PP'$, the $X_n$, $n\in \mathbb{N}$, will become independent equiprobable signs. Transposing from $\mathbb{N}$ to $\ZZn$ it will yield Theorem~\ref{theorem} (recall $(\bullet_2)$ from the Introduction). 


\begin{lemma}
Let $(X,\HH,\Theta)$ be a probability space, $N\in \mathbb{N}$ and $(S_n)_{n=1}^N$ a partition of $X$ into $\Theta$-saturated subsets (saturated: inner measure zero, outer measure one; in particular, not-$\overline{\Theta}$-measurable). Then $\Theta$ admits an extension to a probability $\Theta'$ on $\HH\lor \sigma_X(\{S_1,\ldots,S_N\})$ rendering each $S_i$ independent of $\HH$ and having $\Theta'(S_i)=1/N$, $i\in [N]$. 
\end{lemma}
\begin{proof}
See \cite[p.~139, proof of Example~7.7]{wise1993counterexamples}: it is stated there on Euclidean space for a probability on the Borel sets equivalent to Lebesgue measure, but actually the equivalence condition is only used with reference to \cite[Example~6.9]{wise1993counterexamples} for the existence of the partition, while the rest of the argument is seen easily not to depend on any special property that  Euclidean space with its Borel $\sigma$-field might have viz.  any other measurable space. 
 \end{proof}

Because of the preceding lemma (with $N=2$), to see the existence of the advertised  $\PP'$ it will be enough to show that the event $\{X_1=1\}$ is a saturated  set of $\overline{\PP}$, i.e. that it is of inner measure $0$ and outer measure $1$. To this end note first that the map that ``flips'' the first coordinate is a measure-preserving bimeasurable bijection of $\Omega$ to itself that sends  $\{X_1=1\}$ to $\{X_1=-1\}=\Omega\backslash \{X_1=1\}$. In consequence it is enough to check that $\{X_1=1\}$ has inner measure $0$. Suppose per absurdum that an $A\subset \{X_1=1\}$ has strictly positive $\PP$-measure. 

Let $\star$ be the operation of coordinate-wise multiplication on $\Omega$.  For $\{A,B\}\subset 2^\Omega$ put $A\star B:=\{a\star b:(a,b)\in A\times B\}$, also $k\star A=\{k\star a:a\in A\}$ for $k\in \Omega$ and $A\subset \Omega$ --- such usage of $\star$ is clearly commutative and associative in the clear meaning of these qualifications.

We will establish in a lemma below that  $\{X_1=1\}\star \{X_1=1\}$ contains $\{\xi_1=1,\ldots,\xi_n=1\}$ for some $n\in \mathbb{N}$ (it is a version of the Steinhaus property for the Lebesgue measure). But this cannot be. Notice in fact that if $\{\omega_1,\omega_2\}\subset \{X_1=1\}$ with $\omega_1\sim \omega_2$, then $\omega_1\star\omega_2\in \{X_1=1\}$ [for: because $\omega_1\sim\omega_2$, there is an $n\in \mathbb{N}$ such that $\omega_1$ and $\omega_2$ agree on $\mathbb{N}_{\geq n}$, in particular $X_n(\omega_1)=X_n(\omega_2)$ and, since $\omega_1\star\omega_2$ agrees with $\mathbbm{1}_\mathbb{N}$ on $\mathbb{N}_{\geq n}$, also $X_n(\omega_1\star\omega_2)=X_n(\mathbbm{1}_\mathbb{N})=1$; then $1=X_1(\omega_1)=\xi_1(\omega_1)\cdots\xi_{n-1}(\omega_1)X_n(\omega_1)$ and $1=X_1(\omega_2)=\xi_1(\omega_2)\cdots\xi_{n-1}(\omega_2)X_n(\omega_2)$; therefore $1=\xi_1(\omega_1)\cdots\xi_{n-1}(\omega_1)X_n(\omega_1)\cdot \xi_1(\omega_2)\cdots\xi_{n-1}(\omega_2)X_n(\omega_2)=\xi_1(\omega_1)\xi_1(\omega_2)\cdots\xi_{n-1}(\omega_1)\xi_{n-1}(\omega_2)=\xi_1(\omega_1\star\omega_2)\cdots \xi_{n-1}(\omega_1\star\omega_2)=\xi_1(\omega_1\star\omega_2)\cdots \xi_{n-1}(\omega_1\star\omega_2)X_n(\omega_1\star\omega_2)=X_1(\omega_1\star\omega_2)$]. Further, the $\omega\in \Omega$ that has $\omega_k=(-1)^{\delta_{k,n+1}}$ for all $k\in \mathbb{N}$  belongs to $\{\xi_1=1,\ldots,\xi_n=1\}\cap\{X_1=-1\}$. We must have $\omega=\omega_1\star\omega_2$ for some $\{\omega_1,\omega_2\}\subset \{X_1=1\}$. However, since $\omega\sim \mathbbm{1}_\mathbb{N}$, it means that $\omega_1\sim \omega_2$ and hence $\omega=\omega_1\star\omega_2\in \{X_1=1\}$, a contradiction.

It remains to establish the following version of the Steinhaus theorem.

\begin{lemma}
Let $A$ have positive $\overline{\PP}$-measure. Then $A\star A$ contains a neighborhood of $\mathbbm{1}_\mathbb{N}$.
\end{lemma}
\begin{proof}
It is nearly verbatim the proof of the usual Steinhaus theorem for Lebesgue measure (and actually even a little easier in places). We note that for each $k\in \Omega$, $(\Omega\ni\omega\mapsto \omega\star k\in \Omega)$ is both a measure-preserving bimeasurable bijection and a homeomorphism.

Let $K$ be compact and $U$ be open such that $K\subset A\subset U$ and $2\PP(K)>\PP(U)$; they exist because of the inner and outer regularity of $\overline{\PP}$ (inherited from the same property for the Lebesgue measure via the continuous  mod-$0$ isomorphism $\Phi$). For each $k\in K\subset U$ there is an open neighborhood $W_k$ of $\mathbbm{1}_\mathbb{N}$ of the form $\{\xi_1=\cdots= \xi_n=1\}$ (for some $n\in \mathbb{N}$) such that $k\star W_k\subset U$; note that $W_k\star W_k=W_k$. Then $\{k\star W_k:k\in K\}$ is an open cover of $K$; there is a finite subcover $\{k_1\star W_{k_1},\ldots, k_n\star W_{k_n}\}$ for some $k_1,\ldots,k_n$ from $K$ and $n\in \mathbb{N}$. Put $W:=W_{k_1}\cap \cdots \cap W_{k_n}$, an open neighborhood of $\mathbbm{1}_\mathbb{N}$.

We see that
$$K\star W\subset (\cup_{i=1}^n k_i\star W_{k_i})\star W\subset \cup_{i=1}^n k_i\star W_{k_i}\star W_{k_i}=\cup_{i=1}^n k_i\star W_{k_i}\subset U.$$
Let $w\in W$ and suppose $(K\star w)\cap K=\emptyset$. Then $2\PP(K)=\PP(K\star w)+\PP(K)\leq \PP(U)$,  a contradiction. It means that for every $w\in W$ we have $\{k_1,k_2\}\subset K\subset A$ such that $w\star k_1=k_2$, i.e. $w=k_1\star k_2$, whence $w\in K\star K$. So $W\subset K\star K\subset A\star A$. 
\end{proof}

As a final remark to this proof, notice that now that it has been established that $\{X_1=1\}$ has inner measure zero and outer measure one, the argument supplying the non-$\overline{\PP}$-measurability of $X_1$ becomes superfluous. Still it was quite natural to check the preceding first before attempting the nevertheless more elaborate proof of the saturatedness of $\{X_1=1\}$. 

Let us close this section by spending a little time on a complement to Theorem~\ref{theorem}, namely an analogue of it in which the random variables have diffuse laws.
\begin{remark}\label{continuous-space-analogue}
Consider the SDE with state space $S:=\{-1,1\}^\mathbb{N}$:
\begin{equation}\label{SDE:2}
Y_n=Y_{n-1}\star \eta_n,\quad n\in \mathbb{Z}_{\leq 0},
\end{equation}
where the $\eta_n$, $n\in \mathbb{N}$, are independent uniform (i.e. having law $(\frac{1}{2}\delta_1+\frac{1}{2}\delta_{-1})^{\times \mathbb{N}}$) on $S$,  and the $S$-valued process $(Y_n)_{n\in \ZZn}$ is to be solved for. Still $\star$ is coordinate-wise multiplication.

Note: a probability on $S$ (with the $\sigma$-field $(2^{\{-1,1\}})^{\mathbb{N}}$, of course) is uniform iff it is invariant under $\star$-multiplication (by constants). It follows easily that: (i) for all $n\in \mathbb{N}$, if under some probability the $U_1,\ldots, U_n$ are independent uniform on $S$, then so too are their running $\star$-products $U_1,\ldots, U_1\star \cdots\star  U_n$; (ii) if  $U$ and $V$ are $S$-valued and independent, one of which is uniform on $S$, then $U\star V$ is uniform on $S$ also. 

Suppose now we are given a weak solution to \eqref{SDE:2}, namely, on some filtered probability space, a pair  $(\eta,Y)$ of adapted $S$-valued processes such that \eqref{SDE:2} holds and such that for each $i\in\ZZn$, $\eta_i$ is uniform on $S$  and independent of $\FF_{i-1}$. Then from the observations (i)-(ii) preceding: 
the $Y_n$, $n\in \ZZn$, are independent uniform on $S$ (in particular they are diffuse); $Y_n$ is independent of $\FF^\eta_n$ (even of $\FF^\eta$) for each $n\in\ZZn$. 
There is thus uniqueness in law and no weak solution to \eqref{SDE:2} is strong. On the other hand, starting with a sequence  $Y=(Y_n)_{n\in \ZZn}$ consisting of independent random variables uniform on $S$ we construct at once a weak solution to \eqref{SDE:2}, just like it was done with \eqref{main}.  Moreover, to construct a weak solution to \eqref{SDE:2} that is ``non-anticipative'' in the noise $\eta$ one need simply take the product of $\mathbb{N}$ copies of $(\PP',X)$ as constructed above and proceed in the obvious manner. 
Therefore the phenomenon for \eqref{main} recorded in Theorem~\ref{theorem} persists in a setting with diffuse laws. 
\end{remark}

\section{Theorem~\ref{thm:kolmo}: a converse to Kolmogorov's zero-one law}\label{sec:kolmo}
We work in the setting of Theorem~\ref{thm:kolmo}. The equivalence of \ref{Kii} and \ref{Kiii} is by taking complements. 

Let $n\in \mathbb{N}$ and $\pi$ be a transposition (a transposition exchanges two elements, leaving the others unchanged) of $E_n$. Denote by $\theta_\pi^n:\prod_{m\in \mathbb{N}_{\geq n}}E_m\to \prod_{m\in \mathbb{N}_{\geq n}}E_m$ the map given by $\theta_\pi^n(e):=(\pi(e_n),e_{n+1},e_{n+2},\ldots)$ for $e\in \prod_{m\in \mathbb{N}_{\geq n}}E_m$, i.e. $\theta_\pi^n=\pi\otimes (\otimes_{m\in \mathbb{N}_{>n}} \id_{E_m})$. Clearly $\theta_\pi^n$ is a $(\otimes_{m\in \mathbb{N}_{\geq n}} 2^{E_m})$-bimeasurable involutive bijection.  Furthermore, assuming \ref{Ki}, we see that for all  $k\in \mathbb{N}$ and then for all $e_n\in E_n,\ldots, e_{n+k}\in E_{n+k}$, one has $\PP(\xi_n=e_n,\xi_{n+1}=e_{n+1}\ldots,\xi_{n+k}=e_{n+k})=\PP(\xi_n=e_n)\PP(\xi_{n+1}=e_{n+1})\cdots\PP(\xi_{n+k}=e_{n+k})=\frac{\PP(\xi_n=e_n)}{\PP(\xi_n=\pi(e_n))}\PP(\xi_n=\pi(e_n))\PP(\xi_{n+1}=e_{n+1})\cdots\PP(\xi_{n+k}=e_{n+k})=\frac{\PP(\xi_n=e_n)}{\PP(\xi_n=\pi(e_n))}\PP(\xi_n=\pi(e_n),\xi_{n+1}=e_{n+1},\ldots,\xi_{n+k}=e_{n+k})=\frac{\PP(\xi_n=e_n)}{\PP(\xi_n=\pi(e_n))}\PP(\pi(\xi_n)=e_n,\xi_{n+1}=e_{n+1},\ldots,\xi_{n+k}=e_{n+k})$. By an application of Dynkin's lemma we conclude that $((\xi_k)_{k\in \mathbb{N}_{\geq n}})_\star\PP=D_n\cdot \{[\theta^n_\pi((\xi_k)_{k\in \mathbb{N}_{\geq n}})]_\star\PP\}$, where $D_n:=\left(E_n\ni e\mapsto \frac{\PP(\xi_n=e)}{\PP(\xi_n=\pi(e))}\right)\circ \pr_n:\prod_{m\in \mathbb{N}_{\geq n}}E_m\to (0,\infty)$. It implies that the map $\theta_\pi^n$ preserves the $\PP$-law of $(\xi_k)_{k\in \mathbb{N}_{\geq n}}$ up to equivalence, in the sense that 

 \begin{quote} 
 $(\dagger) \qquad ((\xi_k)_{k\in \mathbb{N}_{\geq n}})_\star\PP\sim (\theta^n_\pi)_\star[((\xi_k)_{k\in \mathbb{N}_{\geq n}})_\star\PP]$.
\end{quote}
We will argue that as a consequence \ref{Kii} holds true. 

\begin{lemma}
Let $(X,\HH,\Theta)$ be a probability space and let $\theta=(\theta_i)_{i\in I}$ be a countable family of  measurable involutions of $X$ such that ${\theta_i}_\star\Theta\sim \Theta$ for each $i\in I$. Suppose an $X^*$ is $\Theta$-almost certain. Then there exists a $\Theta$-almost certain $X^{**}$ contained in $X^*$ that is invariant under $\theta_i$ for each $i\in I$ (i.e. $\theta_i(X^{**})=X^{**}$ for all $i\in I$).
\end{lemma}
\begin{proof}
Suppose first $I=\{1\}$; put $\theta:=\theta_1$ for short. Because $\theta_\star\Theta\sim \Theta$, the event $X^{**}:=X^{*}\backslash \theta(X\backslash X^{*})$ is $\Theta$-almost certain. Besides, $\theta(X^{**})=\theta(X^*)\backslash (X\backslash X^{*})\subset X^{*}\backslash \theta(X\backslash X^{*})=X^{**}$. Owing to $\theta$ being involutive it means that in fact $X^{**}=\theta(X^{**})$. 

Let now $I$ be finite and having at least two elements (the case $I=\emptyset$ is trivial), $I=\{1,\ldots,n\}$ for some $n\in \mathbb{N}_{\geq 2}$. Put $\Omega_0^*:=\Omega^*$. By the preceding, inductively, there are $\Theta$-almost certain and nonincreasing: $\Omega_1^{*}\in 2^{ \Omega^*_0}$ invariant under $\theta_1$, $\ldots$, $\Omega_n^{*}\in 2^{ \Omega^*_{n-1}}$ invariant under $\theta_n$; $\Omega_{n+1}^*\in 2^{ \Omega^*_n}$ invariant under $\theta_1$, $\ldots$, $\Omega_{2n}^*\in 2^{ \Omega^*_{2n-1}}$ invariant under $\theta_n$; and so on and so forth. Putting $\Omega^{**}:=\cap_{n\in\mathbb{N}}\Omega_n^*$ it is plain that $\Omega^{**}\in 2^{\Omega^*}$ is $\Theta$-almost certain. Besides, for each $i\in [n]$: $\theta_i(\Omega^{**})\subset \cap_{k\in \mathbb{N}_0}\Omega^*_{i+kn}=\Omega^{**}$; again by involutiveness it means that $\Omega^{**}$ is invariant under $\theta_i$.

Finally, consider $I=\mathbb{N}$. By what we have just shown, inductively, there is a nonincreasing sequence $(\Omega_n^*)_{n\in \mathbb{N}}$ of $\Theta$-almost certain sets  contained in $\Omega^*$ and with $\Omega_n^*$ invariant under $\theta_1,\ldots,\theta_n$ for each $n\in \mathbb{N}$. Therefore $\Omega^{**}:=\cap_{n\in\mathbb{N}}\Omega_n^*$ is $\Theta$-almost certain, contained in $\Omega^*$, and for each $n\in \mathbb{N}$, $\theta_n(\Omega^{**})\subset \cap_{m\in \mathbb{N}_{\geq n}}\Omega_m^*=\Omega^{**}$, whence $\Omega^{**}$ is also invariant under $\theta$. 
\end{proof}


Now, $\Omega^*\in \sigma(\xi)$ means that $\Omega^*=\xi^{-1}(E^*)$ for some $E^*\in \otimes_{m\in \mathbb{N}}2^{E_m}$; $E^*$ is $\xi_\star\PP$-almost certain because $\Omega^*$ is $\PP$-almost certain (by assumption). The number of transpositions of $E_1$ being denumerable, by the preceding lemma applied to $\xi_\star \PP$ and by $(\dagger)$ with $n=1$, there is a $\xi_\star \PP$-almost certain $E^{**}\in 2^{E^*}$ that is invariant under $\theta_\pi^1$ 
for any transposition $\pi$ of $E_1$. 
Therefore $E^{**}=E_1\times \pr_{\mathbb{N}_{\geq 2}}(E^{**})$ and so $\Omega_1^{**}:=\xi^{-1}(E^{**})\in \sigma(\xi\vert_{\mathbb{N}_{\geq 2}})$. Besides, $\Omega_1^{**}$ is $\PP$-almost certain and contained in $\Omega^*$. 
 
Because of  $(\dagger)$ (and the previous lemma) again, we may moreover proceed inductively to define a whole nonincreasing sequence $(\Omega^{**}_n)_{n\in \mathbb{N}}$ of $\PP$-almost certain sets with $\Omega^*\supset \Omega_n^{**}\in \sigma(\xi_{\mathbb{N}_{>n}})$ for each $n\in \mathbb{N}$. Clearly $\Omega^{**}:=\cap_{n\in \mathbb{N}}\Omega^{**}_n$ is $\PP$-almost certain and belongs to $\limsup_{n\to\infty}\sigma(\xi_n)$. Hence \ref{Kii} in fact holds true.

Suppose now \ref{Kiii} valid, $\xi$ ``sufficiently nice'' and, per absurdum, \ref{Ki} false. For some $n\in \mathbb{N}$ and $e\in E_n$, $\PP(\xi_n=e)=0$, so $\{\xi_n=e\}$ must be contained in a $\PP$-negligible event $B$ belonging to $\limsup_{k\to\infty}\sigma(\xi_k)$. But such $B$, because of the ``$\xi$ is sufficiently nice'' condition, will contain also $\{\xi_n=f\}$ for all $f\in E_n\backslash \{e\}$,  hence $\Omega$, a contradiction. This, together with the above, establishes Theorem~\ref{thm:kolmo}.

\begin{remark}\label{remark:ceteris-paribus}
If, ceteris paribus, for some $n\in \mathbb{N}$, the space $E_n$ is not countable, but rather comes equipped with a $\sigma$-field that contains the singletons (and w.r.t. which $\xi_n$ is a random element), then automatically $\PP(\xi_n=e)=0$ for some $e\in E_n$. By the same token as in the preceding paragraph we see that $\{\xi_n=e\}$ is a $\PP$-negligible event from $\sigma(\xi)$ that is contained in no $\PP$-negligible event of $\limsup_{k\to\infty}\sigma(\xi_k)$, provided of course $\xi$ is ``sufficiently nice''. Thus in this case no converse (in the spirit of Theorem~\ref{thm:kolmo})  to Kolmogorov's zero-one law can be hoped for.
\end{remark}

\begin{example}\label{example:suff-nice}
Let $X:=\{-1,1\}^\mathbb{N}\cup \{0_\mathbb{N}\}$ (where $0_\mathbb{N}$ is the constant $0$ on $\mathbb{N}$), let $\eta=(\eta_k)_{k\in \mathbb{N}}$ be the coordinate process on $X$, $\HH:=\sigma(\eta)$, $\Theta:=(\frac{1}{2}\delta_{-1}+\frac{1}{2}\delta_1)^{\times \mathbb{N}}$. The event $\{\eta_1=0\}=\{0_\mathbb{N}\}$ is $\Theta$-negligible. On the other hand, let $X^*$ be any $\Theta$-negligible event. Then a fortiori $X^*\backslash \{0_\mathbb{N}\}$ is $\Theta_{\{-1,1\}^\mathbb{N}}$-negligible. By Theorem~\ref{thm:kolmo} applied to the space $\{-1,1\}^\mathbb{N}$ it follows that $X^*\backslash \{0_\mathbb{N}\}$ is contained in a $\Theta_{\{-1,1\}^\mathbb{N}}$-negligible tail event of $\eta\vert_{\{-1,1\}^\mathbb{N}}$, hence also $X^*$ is contained in a $\Theta$-negligible tail event of $\eta$. Therefore \ref{Kiii} is met but \ref{Ki} fails (for the process $\eta$ on $(X,\HH,\Theta)$ and taking $E_n=\{-1,0,1\}$ for all $n\in \mathbb{N}$). It means that $\eta$ cannot be ``sufficiently nice'' (as it is not). 
\end{example}

\begin{example}\label{ex:indep}
Let $X:=\{-1,1\}^{\mathbb{N}_0}$, let $\eta=(\eta_k)_{k\in \mathbb{N}_0}$ be the coordinate process on $X$, $\HH:=\sigma(\eta)$, $\Theta$ a probability on $\HH$ under which $\eta_0$ is an equiprobable random sign, while conditionally on $\{\eta_0=1\}$ (resp. $\{\eta_0=-1\}$), the sequence $(\eta_k)_{k\in \mathbb{N}}$ is  that of the (additive) increments of a simple non-degenerate random walk $Z=(Z_n)_{n\in \mathbb{N}_0}$ on the integers  (with $Z_0=0$) that drifts to $\infty$ (resp. $-\infty$). The event $X^*:=\{\sup_{n\in \mathbb{N}_0}Z_n=\infty\}\cap \{\eta_0=-1\}$ is $\Theta$-negligible. If it were contained in a negligible tail event $X^{**}$ of $\eta$ (or even just in a negligible event of $\sigma(\eta\vert_\mathbb{N})$), then $X^{**}$ would contain the tail event $\{\sup_{n\in \mathbb{N}_0}Z_n=\infty\}$, however this event is not $\Theta$-negligible (it has indeed probability  a half). By Theorem~\ref{thm:kolmo} it follows that $\eta$ cannot be an independency (as it is not). 
\end{example}

\begin{example}\label{example:last}
For a ``positive'' example, let  $X:=\{-1,1\}^{\mathbb{N}}$, let $\eta=(\eta_k)_{k\in \mathbb{N}}$ be the coordinate process on $X$, $\HH:=\sigma(\eta)$, $\Theta$ a probability on $\HH$ under which $\eta$ is a sequence of independent equiprobable random signs. Let also $Z=(Z_n)_{n\in \mathbb{N}_0}$  be the random walk on the integers whose sequence of (additive) increments is $\eta$, $Z_0=0$.  The event $A:=\{\sup_{n\in \mathbb{N}_0}Z_n=\infty\}$ is $\Theta$-almost certain, but so is $X^*:=A\backslash \{\omega\}$ for any given $\omega\in A$; the first of these is a tail event, while the latter is evidently not (because if it were, then it would have to not contain any $x\in X$ that agrees eventually with $\omega$, whereas in fact every such $x$ that is $\ne \omega$ belongs to $X^*$, and there are many such $x$ /though we only need one/). Nevertheless, by Theorem~\ref{thm:kolmo}, $X^*$ must contain a $\Theta$-almost certain tail event of $\eta$; of course we can make one explicit immediately, namely $X^{**}:=\{\sup_{n\in \mathbb{N}_0}Z_n=\infty\}\backslash \{x\in X:x\text{ agrees eventually with }\omega\}$. 
\end{example}


\bibliographystyle{plain}
\bibliography{Biblio_noise}
\end{document}